\documentclass[10pt]{siamltex}
\usepackage[utf8]{inputenc}
\usepackage{graphicx}
\usepackage{amsmath}
\usepackage{amssymb}
\usepackage{amsfonts}
\usepackage{mathrsfs}
\usepackage{cite}
\usepackage{fancyhdr}
\usepackage{epstopdf}
\usepackage{hyperref}
\usepackage{url}
\usepackage[usenames]{color}
\usepackage{subcaption}
\usepackage{float}
\usepackage[font=small,labelfont=bf,tableposition=top]{caption}
\usepackage{xifthen}
\usepackage{enumerate}

\newtheorem{thm}{Theorem}[section]
\newtheorem{cor}[thm]{Corollary}
\newtheorem{prop}[thm]{Proposition}
\newtheorem{rmk}{Remark}

\def\propautorefname~#1\null{%
  Proposition~#1\null
}

\newcommand{\abs}[1]{\left\lvert #1\right\rvert}

\newcommand{\OV}{\operatorname{\mathsf{V}}}

\newcommand{\OW}{\operatorname{\mathsf{W}}}

\newcommand{\Id}{\operatorname{\mathsf{Id}}}

\newcommand{\Op}{\operatorname{\mathsf{p}}}

\newcommand{\mOV}{\overline{\OV}}
\newcommand{\mOW}{\overline{\OW}}

\newcommand{\wH}{\widetilde{H}}

\newcommand{\Vx}{\mathbf{x}}

\newcommand{\Vy}{\mathbf{y}}

\newcommand{\Vn}{\mathbf{n}}

\newcommand{\R}{\mathbb{R}}

\newcommand{\N}{\mathbb{N}}

\newcommand{\D}{\mathbb{D}}
\renewcommand{\P}{\mathbb{P}}
\renewcommand{\S}{\mathbb{S}}

\newcommand{\dual}[2]{\left\langle #1\,,\,#2\right\rangle}

\newcommand{\norm}[2]{\left\lVert #1\right\rVert_{#2}}
\newcommand{\oL}{\operatorname{\mathcal{L}}}
\definecolor{OliveGreen}{rgb}{0,0.4,0.1}
\definecolor{gray}{rgb}{0.3,0.3,0.3}

\newcommand{\curl}{\operatorname{\underline{\mathrm{curl}}}}

\def\Xint#1{\mathchoice
   {\XXint\displaystyle\textstyle{#1}}%
   {\XXint\textstyle\scriptstyle{#1}}%
   {\XXint\scriptstyle\scriptscriptstyle{#1}}%
   {\XXint\scriptscriptstyle\scriptscriptstyle{#1}}%
   \!\int}
\def\XXint#1#2#3{{\setbox0=\hbox{$#1{#2#3}{\int}$}
     \vcenter{\hbox{$#2#3$}}\kern-.5\wd0}}

\def\dashint{\Xint-}

\numberwithin{table}{subsection}

\title{Closed-Form Exact Inverses of the Weakly Singular and Hypersingular Operators On Disks}
\author{Ralf Hiptmair\thanks{Seminar for Applied Mathematics, ETH Zurich,
              Raemistrasse 101, 8092 Zurich, Switzerland.
              { \tt ralf.hiptmair@sam.math.ethz.ch}.  } 
        \and   Carlos Jerez-Hanckes\thanks{School of Engineering, Pontificia Universidad Cat\'olica de Chile,
              Av. Vicu\~na Mackenna 4860, Macul, Santiago, Chile.
              { \tt cjerez@ing.puc.cl}. This work was partially funded by Chile CORFO Engineering 2030 program 
              through grant OPEN-UC 201603, Conicyt Anillo ACT1417.} 
        \and Carolina Urz\'ua-Torres\thanks{Seminar for Applied Mathematics, ETH Zurich,
              Raemistrasse 101, 8092 Zurich, Switzerland.
              { \tt carolina.urzua@sam.math.ethz.ch}. }.
            The work of this author was supported by ETHIRA grant ETH-04 13-2.
              } 

\begin{document}
\maketitle

\begin{abstract}
We introduce new boundary integral operators which are the exact inverses of 
the weakly singular and hypersingular operators for $-\Delta$ on flat disks. 
Moreover, we provide explicit closed forms for them and prove the continuity 
and ellipticity of their corresponding bilinear forms in the natural Sobolev 
trace spaces. This permit us to derive new Calder\'on-type identities that can
provide the foundation for optimal operator preconditioning in Galerkin boundary 
element methods. 
\end{abstract}

\section{Introduction}
We study the weakly singular and hypersingular Boundary Integral Operators (BIOs) 
arising when solving screen problems in $\R^3$ via Boundary Integral Equations 
(BIEs). In particular, we consider the following singular BIEs on the flat disk 
$\D_a$ of radius $a>0$:
\begin{eqnarray}
\label{eq:IE1}
 (\OV \sigma) (\Vy) &:=&\frac{1}{4 \pi} \int_{\D_a} 
 \frac{\sigma (\Vx) }{ \norm{\Vx - \Vy}{}} d\D_a(\Vx) = g(\Vy), \\
\label{eq:IE2}
   (\OW u) (\Vy) &:=&\frac{1}{4 \pi} \dashint_{\D_a} u(\Vx) 
   \frac{\partial^2 }{\partial n_x \partial n_y}\frac{1}{ \norm{\Vx - \Vy}{}} 
   d\D_a(\Vx) = \mu(\Vy),
\end{eqnarray}
for $\Vy\in\D_a$ and functions $\sigma, g, u, \mu$ in suitable trace Sobolev 
spaces that will be made explicit later on. Here $\norm{\cdot}{}$ stands for the 
standard Euclidean norm and $\partial/\partial n$ for the normal derivative in a 
direction perpendicular to the disk $\D_a$. 

Equations \eqref{eq:IE1} and \eqref{eq:IE2} are connected with the following 
exterior boundary value problem for the Laplacian $-\Delta$ 
\cite{MCL00,STE08,SAS10}: find $U\in H^1_\text{loc}(\R^3\setminus\overline{\D}_a)$ such 
that 
\begin{equation}
\label{eq:bvp}
 \begin{cases}
 -\Delta U = 0 &\text{in }\Omega_a :=\R^{3}\setminus\overline{\D}_a\:, 
  \\ 
 U = g \quad \text{ or } \quad \dfrac{\partial U}{\partial n} = \mu
  &\text{on }\D_a\;,  \\[0.2cm]  U(\Vx) = \mathcal{O} 
 (\Vert \Vx^{-1} \Vert) \: &\text{as} \: \Vert \Vx \Vert \rightarrow \infty,
  \end{cases}
\end{equation}
The derivation of BIEs relies on a fundamental solution and Green's third identity, 
which yields the so-called \emph{integral representation}. The latter allows to 
reconstruct the solution $U$ over the entire domain $\Omega_a$ from boundary data 
$(U_{\vert \D_a}, \frac{\partial U}{\partial n}_{\vert \D_a})$ via \emph{single 
and double layer potentials}. When taking traces, we arrive at BIEs
for the trace jumps across the disk, \eqref{eq:IE1} for $\sigma := 
\left[  \frac{\partial U}{\partial n} \right]_{\D_a}$, and \eqref{eq:IE2} for $u := \left[ U \right]_{\D_a}$, 
in the case of the exterior Dirichlet or Neumann problem respectively. Usually 
\eqref{eq:IE1} is referred to as \emph{weakly singular} BIE, while \eqref{eq:IE2} is 
called \emph{hypersingular}.

This remains true when replacing $\D_a$ with a connected orientable Lipschitz 
manifold with boundary of co-dimension one, a so-called screen $\Gamma\subset\R^{3}
$. A comprehensive theory of the arising BIOs in the framework of Sobolev spaces 
on screens is available  \cite{STE86, STE87, SAS10, BUC03}. 

In \eqref{eq:IE1} and \eqref{eq:IE2}, $\D_a$ may also be replaced with the unit 
2-sphere $\S$. Then solutions $\sigma$ and $u$ will supply the missing boundary 
data for Dirichlet and Neumann boundary value problems for $- \Delta$ in the 
exterior of the unit ball. Moreover, a remarkable Calder\'on identity will 
hold: On $\S$ the BIOs occurring in \eqref{eq:IE1} and \eqref{eq:IE2} regarded as 
linear mappings between the trace spaces $H^{-1/2}(\S)$ and $H^{1/2}(\S)$, (after 
scaling) turn out to be inverses of each other up to a compact perturbation. 
More precisely, on the $2$-sphere we have $\OV\OW=\OW\OV=\frac{1}{4}(\OV^2-\Id)$ 
\cite[Eq.~(3.2.33)]{NED01}. Moreover, 
Dirichlet and Neumann trace spaces, $H^{1/2}(\S)$ and $H^{-1/2}(\S)$, are dual 
to each other.



However, the situation on screens differs from the case of closed surfaces. Indeed, instead of working with the standard Sobolev trace 
spaces as described above, one must use the spaces $\wH^{1/2}(\Gamma)$ and $\wH^{
-1/2}(\Gamma)$ --also denoted as $H^{\pm1/2}_{00}(\Gamma)$ \cite{LIM72}. In this 
setting, Calder\'on identities break down in the case of screens since
the mapping properties of the weakly singular and hypersingular operators 
degenerate, and the double layer operator and its adjoint vanish \cite{STE87}. 
The very same situation is encountered in two dimensions where the role of $
\D_a$ is played by a straight line segment. In 2D, one of the authors jointly 
with J.-C.~N\'ed\'elec managed to find the exact variational inverses of the 
weakly singular and hypersingular BIO established in \cite{JHN10} for the line 
segment. Their applicability as explicit Calder\'on-type preconditioners for 
open arcs was later shown in \cite{HJU14}. This breakthrough in 2D was the 
starting point to find the closed form of the BIOs that satisfy relations 
analogous to the ones established in \cite{JHN10} and find Calder\'on-type 
identities over the unit disk.

The key contribution of this article is to finally state these identities and 
to provide an explicit construction of the exact inverses of the hypersingular 
and the weakly singular BIOs on the unit disk. We call these inverse BIOs 
\emph{modified weakly singular} and \emph{modified hypersingular operators}, 
respectively, because their kernels feature the same structure as those of the 
standards BIOs but incorporate a smooth cut-off function that depends on the 
distance to the boundary of the disk $\partial\D_a$.

In addition, the modified hypersingular operator is given both in terms of a 
finite part integral operator with a special kernel and in terms of a 
variational (weak) form in suitable trace spaces. It turns out that this 
variational form is related to the modified weakly singular operator (introduced 
in Section~\ref{sec:mOVBIO} and \cite{HJU16}) in exactly the same way as the 
weak form of the standard hypersingular operator can be expressed through that 
for the weakly singular operator \cite[Thm.~6.17]{STE08}.

Instrumental in the derivation of this variational form of the inverse of the 
weakly singular operator have been recent yet unpublished results 
by J.-C.N\'ed\'elec \cite{NED13, NER17} also elaborated in the PhD thesis of 
P.~Ramaciotti \cite{RAM16} and in \cite{NER17}. They use so-called projected 
spherical harmonics in order to state series expansions for the kernels of the 
boundary integral operators and their inverses on the disk. We make use of such 
relations and prove them in a different way as will be shown in 
Section~\ref{sec:series}.


\section{Preliminaries}
\label{sec:prem}

\subsection{Notation}
Let $d=1,2,3$. For a bounded domain $K\subseteq\R^{d}$, $C^{m}(K)$, $m\in\N_0$, 
denotes the space of $m$-times differentiable scalar functions on $K$, and, 
similarly, for the space of infinitely differentiable, scalar continuous 
functions we write $C^{\infty}(K)$. Let $L^{p}(K)$ designate the class of $p
$-integrable functions over $K$. Dual spaces are defined in standard fashion 
with duality products denoted by angular brackets $\dual{\cdot}{\cdot}_{K}$.

Let $\mathcal{O} \in \R^d, \, d=2,3$ be open and $s\in \R$. We denote standard 
Sobolev spaces by $H^s(\mathcal{O})$. For positive $s$ and $\mathcal{O}$ 
Lipschitz such that $\mathcal{O} \Subset \tilde{\mathcal{O}}$ for $\tilde{
\mathcal{O}}\in \R^d$ closed, let $\wH^s(\mathcal{O})$ be the space of functions 
whose extension by zero to $\tilde{\mathcal{O}}$ belongs to $H^s(\tilde{
\mathcal{O}})$, as in \cite{JHN10}. In particular, the following duality 
relations hold
\begin{align}
 \wH^{-1/2}(\mathcal{O}) \equiv \left( H^{1/2}(\mathcal{O}) \right)^\prime 
 \qquad\text{and}\qquad  H^{-1/2}(\mathcal{O}) 
 \equiv \left( \wH^{1/2}(\mathcal{O}) \right)^\prime.
\end{align}

Finite part integrals with distributional meaning as in \cite{MCL00} are labeled 
with a dash as in $\dashint$.

\subsection{Geometry}
We focus on the circular disk $\D_a$ with radius $a>0$, defined as $\D_{a}:=\{
\Vx\in\R^{3}: x_{3}=0\text{ and }\norm{\Vx}{}< a \}$. Thus, the volume complement 
domain becomes $\Omega_{a}:=\R^3\setminus\overline{\D}_{a}$. Often, we will omit 
the third coordinate and use the following short polar coordinate notation: $\Vx 
= (r_x \cos \theta_x, r_x \sin \theta_x) \in \D_a$.

\subsection{Variational BIEs on the Disk $\D_a$}
\label{sec:VBIE}

\subsubsection{Weakly Singular Integral Equation}

As in \eqref{eq:IE1}, we consider the following singular integral equation: for 
$g\in H^{1/2}(\D_a)$, we seek a function $\sigma \in \wH^{-1/2}(\D_{a})$ such 
that 
\begin{equation}
 \label{eq:VIE}
 (\OV \sigma ) (\Vy) := \frac{1}{4 \pi} \int_{\D_a} 
 \frac{\sigma (\Vx) }{ \norm{\Vx - \Vy}{}} d\D_a(\Vx) = g(\Vy), \qquad \Vy\in\D_a. 
\end{equation}
The measure $d\D_a(\Vx) $ denotes the surface element in terms of 
$\Vx\in\D_a$, equal to $r_xdr_xd\theta_x$, and 
the unknown $\sigma$ is the jump of the Neumann trace of 
the solution $U$ of the exterior Dirichlet problem in \eqref{eq:bvp}. 

The symmetric variational formulation for \eqref{eq:VIE} is: find 
$\sigma \in \wH^{-1/2}(\D_a)$ such that for $g\in H^{1/2}(\bar{\D
}_a)$, it holds
\begin{equation}\label{eq:bilinV}
 \dual{\OV \sigma }{\psi}_{\D_a} = \frac{1}{4\pi} \int_{\D_a}\int_{\D_a} 
 \frac{\sigma(\Vx)\psi(\Vy)}{\norm{\Vx-\Vy}{}} d\D_a(\Vx) d\D_a(\Vy) = \dual{g
 }{\psi}_{\D_a},  
\end{equation} 
for all $\psi \in \wH^{-1/2}(\D_a)$ \cite[Sect.~3.5.3]{SAS10}.

\subsubsection{Hypersingular Integral Equation}

As in \eqref{eq:IE1}, we consider the following singular integral equation: for 
$\mu \in H^{-1/2}(\D_a)$, we seek a function $u\in\wH^{1/2}(\D_a)$ such that 
  \begin{equation}
    \label{eq:WIE}
   (\OW u) (\Vy) = \frac{1}{4 \pi} \dashint_{\D_a} u(\Vx) 
   \frac{\partial^2 }{\partial n_x \partial n_y}\frac{1}{ \norm{\Vx - \Vy}{}} 
   d\D_a(\Vx) = \mu(\Vy), \qquad \Vy \in \D_a, 
  \end{equation}
where the unknown $u$ is the jump of the Dirichlet trace of the solution $U$
of the exterior Neumann problem \eqref{eq:bvp} \cite[Sect.~3.5.3]{SAS10}. 

Let $v$ be a continuously differentiable function over $\D_a$, and let 
$\tilde{v}$ be an appropriate smooth extension of $v$ into a three-dimensional 
neighborhood of $\D_a$. Let us introduce the \emph{vectorial} surfacic curl 
operator \cite[p.133]{STE08} as 
\begin{equation}
 \curl_{\D_a} v (\Vx) := \Vn (\Vx) \times \nabla \tilde{v}(\Vx),
\end{equation}
with $\Vn (\Vx)$ being the outer normal of $\D_a$ in $\Vx \in \D_a$, and 
$\nabla$ denoting the standard gradient. 

\begin{prop}
A symmetric variational formulation for \eqref{eq:WIE} is given by: given $\mu
\in H^{-1/2}(\D_a)$, seek $u\in\wH^{1/2}(\D_a)$ such that 
\begin{equation}\label{eq:bilinW}
 \dual{\OW u}{v}_{\D_a} := 
 \frac{1}{4 \pi} \int_{\D_a} \int_{\D_a} \frac{\curl_{\D_a}u(\Vy)\cdot 
 \curl_{\D_a}v(\Vx)}{\norm{\Vx - \Vy}{}} 
 d\D_a(\Vx) d\D_a(\Vy) = \dual{\mu}{v}_{\D_a}
\end{equation} 
for all $v \in \wH^{1/2}(\D_a)$.
\end{prop}
\begin{proof}
 The proof follows the same steps as \cite[Thm.~6.17]{STE08} for closed surfaces. 
 Since $u, v \in \wH^{1/2}(\D_a)$, when integrating by parts, 
 the boundary term vanishes and \cite[Lemma~6.16]{STE08} still holds.
\end{proof}

 Existence and uniqueness of solution of problems \eqref{eq:bilinV} and 
 \eqref{eq:bilinW} was proved by Stephan in \cite[Thm.~2.7]{STE87}. Moreover, 
 for a screen $\Gamma$, the bilinear forms in \eqref{eq:bilinV} and 
 \eqref{eq:bilinW} are continuos and 
 elliptic in $\wH^{-1/2}(\Gamma)$ and $\wH^{1/2}(\Gamma)$, respectively 
 (\emph{cf.} \cite[Thm.~3.5.9]{SAS10}). One can show that in these cases and 
 for sufficiently smooth screens $\Gamma$, when approaching the edges 
 $\partial\Gamma$, the solutions decay like the square-root of 
 the distance to $\partial\Gamma$ \cite{CDD03}.

\section{New Boundary Integral Operators}
\subsection{Modified Weakly Singular Integral Operator}
\label{sec:mOVBIO}
We define the \emph{modified} weakly singular operator as the improper integral
  \begin{equation}
    \label{eq:mV}
    (\mOV \upsilon) (\Vx) := -\frac{2}{\pi^2} \int_{\D_a} \upsilon(\Vy)
\frac{S_a(\Vx, \Vy)}{\norm{\Vx - \Vy}{}} d\D_a(\Vy), \quad \Vx\in\D_a,
  \end{equation}
with the bounded function on $\D_a\times\D_a$
\begin{equation}
  \label{eq:Sfunction}
  S_a(\Vx,\Vy) := \tan^{-1} \left( \frac{\sqrt{a^2-r_x^2 
  \vphantom{r_y^2}}\sqrt{a^2-r_y^2}}{a \norm{\Vx-\Vy}{}} \right), \quad \Vx \neq \Vy.
\end{equation}
We remark that the standard weakly singular BIO (as defined in \eqref{eq:VIE}) 
is given by \eqref{eq:mV} without $S_a(\Vx, \Vy)$ and scaled by $\frac{\pi}{2}$. 
Furthermore, since $\lim_{\Vx \rightarrow \Vy}S_a(\Vx, \Vy)=\frac{\pi}{2}$ 
when $\Vx,\,\Vy\in\D_{a}$, the kernels of $\mOV$ and $\OV$ have 
the same weakly singular behavior in (the interior of) $\D_a$. Also note that 
$S_a(\Vx,\Vy) = 0$ if $\abs{\Vx}=a$ or $\abs{\Vy}=a$. As a consequence, $S_a$, though 
bounded, will be discontinuous on $\partial \D_a \times \partial \D_a$.

\subsection{Modified Hypersingular Singular Integral Operator}
\label{sec:mOWBIO}

We define the \emph{modified} hypersingular operator $\mOW$ through the finite 
part integral
\begin{equation}
 \label{eq:mW}
 (\mOW g)(\Vx):= -\frac{2}{\pi^2}\dashint_{\D_a} g(\Vy)K_{\mOW}(\Vx,\Vy)d\D_a(\Vy),
 \quad \Vx \in \D_a,
\end{equation}
with
\begin{equation}
 \label{eq:mWkernel}
 K_{\mOW}(\Vx,\Vy):=\frac{a}{\norm{\Vx-\Vy}{}^2 \sqrt{\smash[b]{a^2-r_x^2}}
 \sqrt{\smash[b]{a^2-r_y^2}}} + \frac{S_a(\Vx,\Vy)}{\norm{\Vx-\Vy}{}^3},
 \quad \Vx\neq\Vy.
\end{equation}

Since $\lim_{\Vx \rightarrow \Vy} S_a(\Vx, \Vy) = \frac{\pi}{2}$ when $\Vx,
\,\Vy\in\D_a$, the kernel of the standard hypersingular operator $\OW$ (defined 
as in \eqref{eq:WIE}) and the second term in \eqref{eq:mWkernel} have the same 
$\mathcal{O}(\norm{\Vx-\Vy}{}^{-3})$-type singularity in the interior of $\D_a$. 
On the other hand, the first term in \eqref{eq:mWkernel} represents a strongly 
singular kernel in the interior of $\D_a$ and features a hypersingular behavior 
when $\Vx=\Vy\in\partial\D_a$. From these observations we point out that $\mOW$ 
has a truly hypersingular kernel.

\subsection{Calder\'on-type Identities}

The next fundamental result reveals why we are interested in these exotic looking 
integral operators (\emph{cf}.~\cite[Prop.~3.6]{JHN10} or \cite[Thm.~2.2]{HJU14}).\\
\newline
\noindent\fbox{\parbox{0.975\textwidth}{%
\begin{thm}\label{thm:caldId}
The following identities hold:
\begin{align}
 \mOW\OV = \Id_{\wH^{-1/2}(\D_a)}\label{eq:invVIE}, \qquad &\OV\mOW = \Id_{H^{
 1/2}(\D_a)},\\
 \mOV\OW = \Id_{\wH^{1/2}(\D_a)}\label{eq:invWIE}, \qquad &\OW\mOV = \Id_{H^{
 -1/2}(\D_a)}.
\end{align}
\end{thm}
\mbox{$$}}}
\vspace{0.25cm}

Key tools for the proof of Theorem \ref{thm:caldId} are some auxiliary results 
by Li and Rong \cite{LIR02}. First, define the function $\Op(\rho,\theta)$ as
\begin{equation}
  \label{eq:p}
  \Op(\rho, \theta) := \frac{1}{2 \pi} \sum_{n = -\infty}^{\infty}
  \rho^{\abs{n}}\mathrm{e}^{in\theta} = \frac{1}{2 \pi}
  \frac{1-\rho^2}{1+\rho^2-2\rho\cos\theta}, \quad \forall\;\abs{\rho} < 1,
\end{equation}
with $\theta \in [0, 2\pi]$ (\emph{cf.}~\cite[Chap.~1.1]{FAB91}). 
%

This function allows us to rewrite the kernel of $\mOV$ as will be shown in the following
Lemma.
\begin{lemma}[Lemma 1 \cite{LIR02}]
\label{lemm:Rrepresentation}
Let us consider points $\Vx$, $\Vy$ on the disk $\D_a$, satisfying $\Vx \neq \Vy$, whose 
polar coordinates are given by $\Vx=(r_x\cos\theta_x, r_x\sin\theta_x)\in \D_a$ and 
equivalently for $\Vy$. Then, for a parameter $\alpha \in (0,4)$, such that $\alpha 
\neq 2$, it holds
\begin{equation}
\begin{split}
 \frac{1}{4 \pi}\frac{1}{\norm{\Vx - \Vy}{}^\alpha} 
  & = \frac{1}{\pi} \sin{\frac{\alpha \pi}{2}} \dashint_0^{\min(r_x,r_y)}\!\!\!
  \frac{s^{\alpha-1}\Op\left(\frac{s^2}{r_xr_y}, \theta_x - \theta_y\right)}{
  (r_x^2-s^2)^{\alpha/2}(r_y^2-s^2)^{\alpha/2}} ds\\
 \label{eq:Rmax} & = { \frac{1}{\pi}} \sin{\frac{\alpha \pi}{2}}
  \dashint_{\max(r_x,r_y)}^\infty \frac{s^{\alpha-1} \Op\left(\frac{r_xr_y}{s^2}, 
  \theta_x-\theta_y\right)}{(s^2-r_x^2)^{\alpha/2}(s^2-r_y^2)^{\alpha/2}} ds.
\end{split}
\end{equation}
Here the integrals above are understood in the sense of finite-part
integrals if $\alpha>2$.
\end{lemma}

We can now follow Fabrikant \cite[Chap.~1.1]{FAB91} and introduce
\begin{equation}
\begin{split}
L(\rho)u(r,\theta) &:= \int_0^{2\pi} \Op(\rho,\theta - \theta_0)u(r,\theta_0) 
d\theta_0 \\ &= \frac{1}{2 \pi} \sum_{n = -\infty}^{\infty} \rho^{\abs{n}}
\mathrm{e}^{in\theta} \int_0^{2\pi} \mathrm{e}^{-in\theta_0}u(r,\theta_0) d \theta_0, 
\end{split}
\end{equation}
with $\theta \in [0, 2\pi]$, and $r\in[0,a]$. This integral operator is sometimes called 
Poisson integral over the disk \cite[Chap.~1.1]{FAB91}. The properties of $L(\rho)
u(r,\theta)$ combined with the formulas from Lemma~\ref{lemm:Rrepresentation}
lead to a complete separation of variables. This fact plays a key role as the resulting 
expressions for $\OV$ and $\OW$ will be iterative systems of Abel integral equations, 
whose solutions are given in the next Theorem.
%
\begin{thm}[Thm.1-2 \cite{LIR02}] 
\label{thm:inverse}
 Let $\alpha\in\lbrace1,3\rbrace$ and $f\in C^1(\overline{\D}_a)$. Then, 
the solution of
\begin{equation}\label{eq:IE}
 \frac{1}{4 \pi} \int_{\D_a} \frac{ \varphi(\Vx)}{\norm{\Vx-\Vy}{}^{\alpha}} 
 d\D_a(\Vx) = f(\Vx), \qquad  \Vx\in\D_a,
\end{equation}
is given by
\begin{equation}\label{eq:IEsol}
 \varphi(\Vx) = \frac{1}{\pi}\dashint_{\D_a}\frac{f(\Vy)}{R_{\D}^{4-\alpha}}d\D_a(\Vy),                 
\end{equation} 
with
\begin{equation}
 \frac{1}{R_\D^{4-\alpha}} = 4\sin\left(\frac{\alpha\pi}{2}\right) \dashint_{\max(r_x, r_y)}^a
 \frac{s^{3-\alpha} \Op\left( \frac{r_xr_y}{s^2}, \theta_x-\theta_y \right)}{
 (s^2-r_x^2)^{(4-\alpha)/2}(s^2-r_y^2)^{(4-\alpha)/2}} ds.
\end{equation}
\end{thm}

\begin{rmk}
From \eqref{eq:Rmax} we notice that ${R_{\D}(\Vx, \Vy)}$ is a scaled 
restriction of ${\norm{\Vx - \Vy}{}}$ from $\R^3$ to $\D_a$. Moreover, for 
$a=\infty$, Theorem~\ref{thm:inverse} implies $\frac{1}{R_{\D}(\Vx,\Vy)} 
=\frac{1}{\norm{\Vx-\Vy}{}}$.
\end{rmk}

\begin{lemma}
\label{lem:atanformula}
Let $a>0$ and $\Vx,\Vy\in\D_a$. If $a \geq s\geq \max(r_x,r_y)$,
 we find the following primitive
\begin{align}
\label{eq:prim}
 \int \frac{s^2 \Op\left(\frac{r_xr_y}{s^2}, \theta_x - \theta_y\right)}{(s^2-
 r_x^2\vphantom{r_y^2})^{3/2}(s^2-r_y^2)^{3/2}} ds & = -\frac{1}{2\pi} \left(
 \frac{s}{\norm{\Vx-\Vy}{}^2\sqrt{\smash[b]{s^2-r_x^2}}\sqrt{\smash[b]{s^2-
 r_y^2}}} \right.\nonumber \\ & \left. \qquad \qquad + \frac{\tan^{-1}\left(
 \frac{\sqrt{s^2-r_x^2\vphantom{r_y^2}}\sqrt{s^2-r_y^2}}{s\norm{\Vx-\Vy}{}}
 \right)}{\norm{\Vx-\Vy}{}^3} \right).
\end{align}
\end{lemma}
\begin{proof}
This can be shown by direct calculation using the following change of variable 
\cite{FAB91}:
\begin{equation*}\label{eq:cov}
 \eta := \frac{\sqrt{\smash[b]{s^2-r_x^2}}\sqrt{\smash[b]{s^2-r_y^2}}}{s},\qquad 
 \frac{d\eta}{d s} = \frac{s^4 - r_x^2r_y^2}{\eta s^3},
\end{equation*}
which leads to
\begin{align*}
 \int \frac{s^2}{(s^2-r_x^2\vphantom{r_y^2})^{3/2}(s^2-r_y^2)^{3/2}} \frac{1-
 r_x^2r_y^2}{1+r_x^2r_y^2-2r_xr_y\cos(\theta_x-\theta_y)} ds = 
 \int \frac{\eta^{-2}}{\norm{\Vx-\Vy}{}^2 + \eta^2} d\eta,
\end{align*}
where
\begin{equation}
 \int \frac{\eta^{-2}}{\norm{\Vx-\Vy}{}^2 + \eta^2} d\eta = - \frac{1}{\eta 
 \norm{\Vx-\Vy}{}^2} - \frac{\tan^{-1}\left(\frac{\eta}{\norm{\Vx-\Vy}{}}
 \right)}{\norm{\Vx-\Vy}{}^3}.
\end{equation}
By definition of  $\eta$ the result follows.
\end{proof}

Combining the above elements we can prove the next result.
\begin{prop} \label{prop:Vinvcont}
The solution of the  weakly singular integral equation \eqref{eq:VIE} can be 
written as  $\sigma(\Vx) = (\mOW g)(\Vx)$, for all $\Vx \in \D_a$, 
if $g$ is continuously differentiable.
\end{prop}
\begin{proof}
Applying Theorem~\ref{thm:inverse}, we get that the solution to \eqref{eq:VIE} 
can be written as \eqref{eq:IEsol}. Moreover, when $a < \infty$, we may use
Lemma~\ref{lem:atanformula} and write
\begin{align*}
 \frac{1}{\pi} \frac{1}{R_{\D}^3(\Vx, \Vy)}  
 &=  {\frac{4}{\pi}} \dashint_{\max(r_x,r_y)}^a \frac{s^2}{(s^2-r_x^2)^{3/2}(s^2
 -r_y^2)^{3/2}} \Op\left(\frac{r_xr_y}{s^2},\theta_x-\theta_y\right) ds\\
 &= -{\frac{2}{\pi^2}} {\footnotesize{fp}}\left(\frac{s}{\norm{\Vx-\Vy}{}^2
 \sqrt{\smash[b]{s^2-r_x^2}}\sqrt{\smash[b]{s^2-r_y^2}}} \right. \\ 
 &\qquad  \qquad \quad \left. + \frac{\tan^{-1}\left(
 \frac{\sqrt{s^2-r_x^2\vphantom{r_y^2}}\sqrt{s^2-r_y^2}}{s\norm{\Vx-\Vy}{}}
 \right)}{\norm{\Vx- \Vy}{}^3}\right) \biggr\vert_{\max(r_x,r_y)}^a,
\end{align*}
where the finite part (\emph{fp}) of the last expression needs to be considered. 
This means that we drop the term corresponding to evaluating our primitive 
\eqref{eq:prim} on the lower bound $\max(r_x,r_y)$, as it becomes infinite.

Hence, we obtain
\begin{align*}
\frac{1}{\pi} \frac{1}{R_{\D}^3(\Vx, \Vy)}  
 &= -\frac{2}{\pi^2} \left(\frac{a}{\norm{\Vx-\Vy}{}^2\sqrt{\smash[b]{a^2-
 r_x^2}}\sqrt{\smash[b]{a^2-r_y^2}}} + \frac{S_a(\Vx,\Vy)}{\norm{\Vx-\Vy}{}^3}
 \right) \\
 &= -\frac{2}{\pi^2} K_{\mOW}(\Vx,\Vy) ,
\end{align*}
with $K_{\mOW}$ from \eqref{eq:mWkernel} as stated.
\end{proof}

\begin{lemma}[Eq.~1.2.14 in \cite{FAB91}]
\label{lem:atanformula2}
Let $a>0$ and $\Vx,\Vy\in\D_a$. If $a \geq s\geq \max(r_x,r_y)$,
 we find the following primitive
\begin{equation}
\begin{split}
\label{eq:prim2}
 \int \frac{1}{(s^2-r_x^2)^{1/2}(s^2-r_y^2)^{1/2}} & \Op\left(\frac{r_xr_y}{s^2},
 \theta_x-\theta_y\right) ds \\&= \frac{1}{2 \pi} \frac{1}{\norm{\Vx-\Vy}{}}
 \tan^{-1} \left(\frac{\sqrt{s^2-r_x^2\vphantom{r_y^2}}\sqrt{s^2-r_y^2}}{s\norm{
 \Vx-\Vy}{}}\right).
\end{split}
\end{equation}
\end{lemma}
\begin{proof}
This can be shown by using the same change of variables in \eqref{eq:cov} and 
direct calculation.
\end{proof}

\begin{prop} \label{prop:Winvcont}
 When $\mu$ is continuosly differentiable, the solution of the 
 hypersingular integral equation \eqref{eq:WIE} can be written as
 $u(\Vx) = (\mOV\mu)(\Vx)$, for all $\Vx \in \D_a$.
\end{prop}
\begin{proof}
Using Theorem~\ref{thm:inverse}, we get that the solution to \eqref{eq:WIE} can be 
written as \eqref{eq:IEsol}. Moreover, when $a < \infty$, we may use 
Lemma~\ref{lem:atanformula2}, to write
\begin{align*}
 \frac{1}{\pi} \frac{1}{R_{\D}(\Vx,\Vy)}  
 &=  -{\frac{4}{\pi}} \dashint_{\max(r_x,r_y)}^a \frac{1}{(s^2-r_x^2)^{1/2}(s^2-
 r_y^2)^{1/2}}\Op\left(\frac{r_xr_y}{s^2}, \theta_x-\theta_y\right) ds\nonumber\\
 &= -{\frac{2}{\pi^2}} \frac{1}{\norm{\Vx-\Vy}{}} \tan^{-1}\left(\frac{\sqrt{s^2-
 r_x^2 \vphantom{r_y^2}}\sqrt{s^2-r_y^2}}{s\norm{\Vx-\Vy}{}}\right) 
 \biggr\vert_{\max(r_x,r_y)}^a  \nonumber\\
 &= -{\frac{2}{\pi^2}} \frac{1}{\norm{\Vx-\Vy}{}} \biggr\lbrace \tan^{-1}\left(
 \frac{\sqrt{a^2-r_x^2 \vphantom{r_y^2}}\sqrt{a^2-r_y^2}}{a \norm{\Vx-\Vy}{}}
 \right) -\tan^{-1}(0) \biggr\rbrace  \nonumber\\
 &= -\frac{2}{\pi^2}\frac{1}{\norm{\Vx-\Vy}{}}\tan^{-1}\left(\frac{ 
 \sqrt{a^2-r_x^2 \vphantom{r_y^2}}\sqrt{a^2-r_y^2}}{a \norm{\Vx-\Vy}{}} \right)
 =  -\frac{2}{\pi^2}\frac{S_a(\Vx,\Vy)}{\norm{\Vx-\Vy}{}},
\end{align*}
as stated with $S_a$ from \eqref{eq:Sfunction}.
\end{proof}

With the above elements, we can now prove our Calder\'on-type Identities:

\emph{Proof of Theorem~\ref{thm:caldId}.} Both identities in \eqref{eq:invVIE} follow from Proposition~\ref{prop:Vinvcont} 
 combined with the density of $C^\infty(\overline{\D}_a)$ in $H^{-1/2}(\D_a)$.
 Analogously, the relations in \eqref{eq:invWIE} are consequences of 
 Proposition~\eqref{prop:Winvcont} and the density of $C^\infty(\overline{\D}_a)$ in 
 $H^{1/2}(\D_a)$.
\endproof\\
~\\
\noindent\fbox{\parbox{0.975\textwidth}{%
\begin{prop}
The operators
\begin{align}
\mOV \, : \, H^{-1/2}(\D_a) \rightarrow \wH^{1/2}(\D_a)
\quad\text{and}
\quad
 \mOW \, : \, H^{1/2}(\D_a) \rightarrow \wH^{-1/2}(\D_a)
\end{align}
are continuous.
\end{prop}
\mbox{$$}}}
\vspace{0.25cm}
\begin{proof}
 We begin our proof with $\mOV \, : \, H^{-1/2}(\D_a) \rightarrow \wH^{1/2}(
 \D_a)$. Let us assume that $\mOV$ is not a bounded operator. Then, by virtue 
 of density, there exists a sequence $(\mu_n)_n\in C^\infty(\overline{\D}_a)$ such 
 that
 \begin{equation}
  \norm{\mu_n}{H^{-1/2}(\D_a)} = 1, \qquad \norm{\mOV \mu_n}{\wH^{1/2}(\D_a)} 
  \rightarrow \infty, \text{ as } n \rightarrow \infty.
 \end{equation}
Since $\OW \, : \, \wH^{1/2}(\D_a) \rightarrow H^{-1/2}(\D_a)$ is an isomorphism 
(\emph{cf.}~\cite[Thm.~2.7]{STE87}), it holds
\begin{equation}
 \norm{\mOV \mu_n}{\wH^{1/2}(\D_a)} \leq C \norm{\OW \mOV \mu_n}{H^{-1/2}(\D_a)} 
 \underset{\eqref{eq:invWIE}}{=} C \norm{\mu_n}{H^{-1/2}(\D_a)},
\end{equation}
from where we get a contradiction. The proof for $\mOW$ is analogous.
\end{proof}
~\\
\newline
\noindent\fbox{\parbox{0.975\textwidth}{%
\begin{cor} \label{cor:ellcont}
 The bilinear forms 
 \begin{align}
 (\vartheta,\mu) \mapsto \dual{\mOV \vartheta}{\mu}_{\D_a}, \qquad &\vartheta, 
 \mu\in H^{-1/2}(\D_a), \\
 (u,g) \mapsto \dual{\mOW u}{g}_{\D_a}, \qquad &u,g\in H^{1/2}(\D_a),
 \end{align}
are elliptic and continuous in  $H^{-1/2}(\D_a)$ and $H^{1/2}(\D_a)$, 
respectively.
\end{cor}
\mbox{$$}}}
\vspace{0.25cm}
\begin{proof}
 Follows from continuity and ellipticity of the standard BIOs $\OW$ and $\OV$ 
 combined with Theorem~\ref{thm:caldId}.
\end{proof}

\section{Bilinear Form for the Modified Hypersingular Integral Operator over $\D_a$}
\label{sec:mOWibp}

We note that formula \eqref{eq:mWkernel} is not practical when implementing a 
Galerkin BEM discretization. On $\S$ and, more generally, on every closed surface 
$\partial\Omega$, we have
\begin{equation}
\dual{\OW u}{v}_{\partial\Omega} = \dual{\OV \curl_{\partial\Omega}u}{\curl_{
		\partial\Omega}v}_{\partial\Omega},
\end{equation}
where, abusing notations, we wrote $\OV$ and $\OW$ for the associated weakly 
singular and hypersingular BIOs (\emph{cf.}~\cite{STE08,SAS10}).

In this section we establish an analogous relation between $\mOW$ and $\mOV$. 
Let us begin by considering the modified weakly singular operator $\mOV$ over 
the unitary disk given by \eqref{eq:mV} and denote $ \omega(\Vx):=\sqrt{1-r_x^2},
 \:\Vx \in \D_1$.

\begin{prop}
 \label{prop:Vinvvar}
The bilinear form associated to the modified hypersingular operator $\mOW$ 
over $\D_1$ satisfies
\begin{equation}
 \label{eq:Vinvvar}
 \dual{\mOW u}{v}_{\D_1} = \frac{2}{\pi^2} \int_{\D_1} \int_{\D_1} \frac{
 S_1(\Vx, \Vy)}{\norm{\Vx-\Vy}{}} \curl_{\D_1,\Vx} u(\Vx)\cdot\curl_{\D_1,\Vy} 
 v(\Vy) d\D_1(\Vx)d\D_1(\Vy),
\end{equation}
for all {\bf $u,v\in H_*^{1/2}(\D_1) := \lbrace v \in H^{1/2}(\D_1) \,:\, \dual{
v}{\omega^{-1}}_{\D_1}=0 \rbrace$}.
\end{prop}

This result was first reported by Ned\'el\'ec and Ramaciotti in their spectral 
study of the BIOs over $\D_1$ and their variational inverses \cite{RAM16}. For 
the sake of completeness, we introduce the key tools they derived and provide 
an alternative simpler proof of this proposition in Section~\ref{sec:series}. 
A proof by means of formal integration by parts remains elusive, as it encounters 
difficulties due to the finite part integrals involved in the definition of 
$\mOW$ and its kernel introduced in \eqref{eq:IEsol}.

We also emphasize that the space $H_*^{1/2}(\D_1)$ corresponds to $H^{1/2}(\D_1
)/\R$ (see end of Section~\ref{sec:series} for further details). It is important to observe that the right-hand side of \eqref{eq:Vinvvar} maps 
constants to zero and thus has a non-trivial kernel if considered in the 
whole $H^{1/2}(\D_1)$ space. In other words, the bilinear form in 
\eqref{eq:Vinvvar} is $H^{1/2}_*(\D_1)$-elliptic but does not have this property 
on $H^{1/2}(\D_1)$. For this reason, its extension to $H^{1/2}(\D_1)$ 
does not actually match the bilinear form of $\mOW$ there, which is $H^{1/2}(\D_1)$-elliptic. In order to remedy this situation, we add an appropriate 
regularizing term coming from the definition of $H^{1/2}_*(\D_1)$ and the 
following result.

\begin{prop}
\label{prop:mOW1}
The following identity holds:
 \begin{equation}
 \label{eq:mOW1}
(\mOW 1)(\Vy)= \frac{4}{\pi}\omega^{-1}(\Vy), \qquad \Vy\in\D_1.
\end{equation}
\end{prop}
\begin{proof}
Let us rewrite the modified hypersingular BIO acting on the constant 
function equal to one as
\begin{align}
(\mOW 1)(\Vy) &= \frac{2}{\pi^2}\int_{\D_1} k_{\mOW}(\Vx,\Vy) d\D_1(\Vx) \nonumber 
\\&= \frac{4}{\pi} \int_{r_y}^1 \frac{s^2}{(s^2-r_y^2)^{3/2}} 
\int_0^s \frac{r_x}{(s^2-r_x^2)^{3/2}} dr_x ds, \label{eq:expW1}
\end{align}
where in the second line \cite[Eq.~(39)]{LIR02} was applied, with an appropriate scaling 
for \eqref{eq:VIE}. Based on Hadamard's finite part integration, we can derive 
the following two formulas \cite[Eq.~(34)]{LIR02}\cite[Eq.~(50)]{FAB97}:
\begin{align}
\frac{d}{d s}\int_0^s \frac{r_x}{\sqrt{s^2-r_x^2}} dr_x = 
-s \int_0^s \frac{r_x}{(s^2-r_x^2)^{3/2}} dr_x,
\end{align}
and
\begin{align}
\frac{d}{d r_y}\int_{r_y}^1 \frac{s f(s)}{\sqrt{\smash[b]{s^2-r_y^2}\vphantom{
			r_x}}} ds &= 
r_y \int_{r_y}^1 \frac{s f(s)}{(s^2-r_y^2)^{3/2}} ds \nonumber\\
&= -\frac{r_yf(1)}{\sqrt{\smash[b]{1-r_y^2}\vphantom{r_x}}} + r_y\int_{r_y}^1 
\frac{1}{\sqrt{\smash[b]{s^2-r_y^2}\vphantom{r_x}}} \frac{d}{ds} f(s) ds.
\end{align}
Using these derivatives in \eqref{eq:expW1}, we obtain
\begin{align*}
(\mOW 1)(\Vy) &= -\frac{4}{\pi} \frac{1}{r_y} \frac{d}{d r_y}\int_{r_y}^1 
\frac{s}{\sqrt{\smash[b]{s^2-r_y^2}\vphantom{r_x}}}\frac{d}{d s}\int_0^s 
\frac{r_x}{\sqrt{s^2-r_x^2}} dr_x ds.
\end{align*}
We integrate the inner integral and get
\begin{align*}
\int_0^s \frac{r_x}{\sqrt{s^2-r_x^2}}dr_x = (-\sqrt{s^2-r_x^2})\big\vert_{0}^s=s.
\end{align*}
Our expression then becomes
\begin{align*}
(\mOW 1)(\Vy) &= -\frac{4}{\pi} \frac{1}{r_y} \frac{d}{d r_y}\int_{r_y}^1 
\frac{s}{\sqrt{\smash[b]{s^2-r_y^2}\vphantom{r_x}}}ds.
\end{align*}
Since
\begin{align*}
\int_{r_y}^1 \frac{s}{\sqrt{\smash[b]{s^2-r_y^2}\vphantom{r_x}}}ds = \sqrt{1-
	r_y^2},
\end{align*}
it holds,
\begin{align*}
(\mOW 1)(\Vy) &= \frac{4}{\pi} \frac{1}{\sqrt{\smash[b]{1-r_y^2}\vphantom{r_x
}}},
\end{align*}
as stated.
\end{proof}

As expected, this result is consistent with known solutions of \eqref{eq:VIE} 
when the right-hand side is $g=1$ \cite{MAR06}. From this, we see that for $u_c$ constant, $(\mOW u_c)(\Vy)$ is equivalent to
 \begin{equation}
(\mOW u_c)(\Vy)= \frac{2}{\pi^2} \int_{\D_1} u_c\omega^{-1}(\Vx)\omega^{-1}
(\Vy) d\D_1(\Vx), \qquad \Vy\in\D_1,
\end{equation}
since $\dual{1}{\omega^{-1}}_{\D_1}=2\pi$. Therefore, the following result holds: \\
~\\
\noindent\fbox{\parbox{0.975\textwidth}{%
\begin{prop}
\label{prop:Vinvvarfull}
 The symmetric bilinear form associated to the modified hypersingular operator
 $\mOW: H^{1/2}(\D_1) \rightarrow \wH^{-1/2}(\D_1)$ can be written as
 \begin{align}
\label{eq:bbmOW}
 \dual{\mOW u}{v}_{\D_1} =& \frac{2}{\pi^2} \int_{\D_1} \int_{\D_1} \frac{
 S_1(\Vx, \Vy)}{\norm{\Vx-\Vy}{}} \curl_{\D_1,\Vx} u(\Vx)\cdot\curl_{\D_1,\Vy} 
 v(\Vy) d\D_1(\Vx)d\D_1(\Vy) \nonumber \\ &+ 
 \frac{2}{\pi^2} \int_{\D_1}\int_{\D_1} \frac{u(\Vx)v(\Vy)}{
 \omega(\Vx)\omega(\Vy)} d\D_1(\Vx)d\D_1(\Vy), \quad \forall u,v\in H^{1/2}(
 \D_1).
\end{align}
\end{prop}
\mbox{$$}}}
\vspace{0.25cm}
\begin{proof}
 Note that we have added the required regularization to \eqref{eq:Vinvvar} such 
 that \eqref{eq:mOW1} is preserved. This guarantees that our bilinear form 
 defined in \eqref{eq:bbmOW} is by construction equivalent to the bilinear form 
 arising from our modified hypersingular operator $\mOW$ \eqref{eq:mW} on $H^{1
 /2}(\D_1)$. Thus, it is $H^{1/2}(\D_1)$-continuous and elliptic. 
\end{proof}\\
~\\
Proposition~\ref{prop:Vinvvarfull} gives us a variational form for $\mOW$ that 
can be easily implemented. Observe that the chosen regularization to extend the bilinear form 
 \eqref{eq:Vinvvar} from $H^{1/2}_*(\D_1)$ to $H^{1/2}(\D_1)$ is analogous 
 to the one needed for the modified hypersingular operator on a segment
 \cite[Eq.~(2.11)]{HJU14}.

In addition, by linearity of the scaling map $\Psi_a:\D_1\rightarrow\D_a$, 
we get the corresponding relationship on $\D_a$.
\newline
\noindent\fbox{\parbox{0.975\textwidth}{%
\begin{cor}
For $a>0$, the symmetric bilinear form associated to the modified hypersingular operator
$\mOW: H^{1/2}(\D_a) \rightarrow \wH^{-1/2}(\D_a)$ can be written as
\begin{align}
\dual{\mOW u}{v}_{\D_a} =& \frac{2}{\pi^2} \int_{\D_a} \int_{\D_a} \frac{
	S_a(\Vx, \Vy)}{\norm{\Vx-\Vy}{}} \curl_{\D_a,\Vx} u(\Vx)\cdot\curl_{\D_a,\Vy} 
v(\Vy) d\D_a(\Vx)d\D_a(\Vy) \nonumber \\ &+ 
\frac{2}{a\pi^2} \int_{\D_a}\int_{\D_a} \frac{u(\Vx)v(\Vy)}{
	\omega_a(\Vx)\omega_a(\Vy)} d\D_a(\Vx)d\D_a(\Vy), \quad \forall u,v\in H^{1/2}(
\D_a),
\end{align}
where $ \omega_a(\Vx) := \sqrt{a^2-r_x^2}, \:\Vx \in \D_a$.
\end{cor}
\mbox{$$}}}
\vspace{0.25cm}


For the standard weakly singular and hypersingular operators, the following 
relation between their kernels holds:
\begin{equation*}
\underbrace{\frac{1}{\norm{\Vx-\Vy}{}^3}}_{=:K_{\OW}(\Vx,\Vy)} = \Delta 
\underbrace{\frac{1}{\norm{\Vx-\Vy}{}}}_{=:K_{\OV}(\Vx,\Vy)}, \qquad \Vx\neq\Vy, \:\Vx,\Vy\in\R^3,
\end{equation*}
while their modified versions do not satisfy this relation. Actually, one has 
for $\Vx\neq\Vy, \:\Vx,\Vy\in\R^3$ \cite[Eq.~(23)]{FAB97}
\begin{align*}
\Delta K_{\mOV}(\Vx,\Vy)&:= -\frac{(2\pi a) \Op\left(\frac{r_xr_y}{a},\theta_x-
\theta_y \right)}{(a^2-r_y^2)^{3/2}\sqrt{a^2-r_x^2}} + \frac{a}{\norm{\Vx-\Vy}{
}^2\sqrt{\smash[b]{a^2-r_x^2}}\sqrt{\smash[b]{a^2-r_y^2}}} + \frac{S_a(\Vx,\Vy)}{
\norm{\Vx-\Vy}{}^3} \\
& = -\frac{(2\pi a) \Op\left(\frac{r_xr_y}{a},\theta_x-\theta_y \right)}{(a^2-
r_y^2)^{3/2}\sqrt{a^2-r_x^2}} + K_{\mOW}(\Vx,\Vy) \neq K_{\mOW}(\Vx,\Vy),
\end{align*}
where the first term is singular on $\partial \D_a$ and is surprisingly not 
symmetric. The interpretation of this term need further investigation.

\begin{rmk}
The representation \eqref{eq:bbmOW} established in Proposition~\ref{prop:Vinvvarfull} 
is essential for the use of $\mOW$ in the context of Galerkin boundary element 
methods, because the direct discretization of hypersingular boundary integral 
operators is not possible without some prior regularization.
\end{rmk}

\subsection{Diagonalization of BIOs on the unit disk $\D_1$}
\label{sec:series}
We begin by introducing key definitions and results on 
the unit disk $\D_1$. When needed, we provide a short proof for the results 
reported in \cite[Chap.~2]{RAM16}.
\begin{definition}[{{\bf Ned\'el\'ec's Projected Spherical Harmonics (PSHs) 
		over $\D_1$} \cite[Eq.~(75)]{NED13}}]
For $l,m \in \N_0$ such that $\abs{m}\leq l$ and $\Vx = (r_x, 
\theta_x)$, we introduce the functions:
\begin{equation}
y_l^m(\Vx) := \gamma_l^m e^{im\theta_x} \P_l^m(\sqrt{1-r_x^2}), \qquad 
\gamma_l^m := (-1)^m \sqrt{\dfrac{(2l+1)}{4\pi} \dfrac{(l-m)!}{(l+m)!}}.
\end{equation}
\end{definition}
The PSH over $\D_1$ satisfy the 
following orthogonality identity \cite[eq.(79)]{NED13}
\begin{equation}
\label{eq:orthylm}
\int_{\D_1} \dfrac{y_{l_1}^{m_1}(\Vy) \overline{y_{l_2}^{m_2}}(\Vy)}{\omega
	(\Vy)} d\D_1(\Vy) = \frac{1}{2} \delta_{l_1}^{l_2} \delta_{m_1}^{m_2},
\end{equation}
where $\delta_{l_1}^{l_2}$ is the Kronecker symbol and $\omega (\Vy) = \sqrt{1-r_y^2}$.
\begin{prop}[\cite{WOL71}]
\label{prop:Wolfe}
The PSHs solve a generalized eigenvalue problem for $\OV$ over $\D_1$ in the 
sense that
\begin{equation}
\label{eq:WolfeV}
\left(\OV \dfrac{y_l^m}{\omega} \right)(\Vx) = \frac{1}{4} \lambda_l^m y_l^m(\Vx),
\quad l+m \text{ even},
\end{equation}
with $ \lambda_l^m := \dfrac{\mathsf{\Gamma}\left(\frac{l+m+1}{2}\right)
\mathsf{\Gamma}\left(\frac{l-m+1}{2}\right)}{\mathsf{\Gamma}\left(\frac{l+m+2}{
2}\right)\mathsf{\Gamma}\left(\frac{l-m+2}{2}\right)}$, and $\mathsf{\Gamma}$ 
being the Gamma function.
\end{prop}

\begin{prop}[\cite{KRE79}]
\label{prop:Krenk}
The PSHs solve a generalized eigenvalue problem for $\OW$ over $\D_1$ in the 
sense that
\begin{equation}
\label{eq:KrenkW}
(\OW y_l^m )(\Vx) = \frac{1}{\lambda_l^m} \dfrac{y_l^m(\Vx)}{
	\omega(\Vx)}, \quad l+m \text{ odd},
\end{equation}
with $\lambda_l^m$ as in Proposition~\ref{prop:Wolfe}.
\end{prop}

\begin{rmk}
Propositions \ref{prop:Wolfe} and \ref{prop:Krenk} nicely show how, in the 
case of a disk, the usual $\OV$ and $\OW$ have reciprocal symbols but 
$\OW\OV\neq\frac{1}{4}\Id$ due to their mapping properties. This is 
characterized here by the parity of the PSHs involved.
\end{rmk}

\begin{prop}[{\cite[Sect.~2.7]{RAM16}
}]
The kernels of $\OV, \OW, \mOV, \mOW$ have the following \emph{symmetric} 
series expansions on $\D_1$:
\small
\begin{eqnarray}
\label{eq:KV} 
K_{\OV}(\Vx,\Vy) &:=& \frac{1}{4\pi} \frac{1}{\norm{\Vx-\Vy}{}} =  
\sum_{l=0}^{\infty}\sum_{\substack{m=-l \\ l+m \text{ even}}}^l \frac{
\lambda_l^m}{4} \left(y_l^m(\Vx)\overline{y_l^m}(\Vy) + \overline{y_l^m}(\Vx) 
y_l^m(\Vy) \right).\\
\label{eq:KmV}
K_{\mOV}(\Vx,\Vy) &:=&
\sum_{l=0}^{\infty}\sum_{\substack{m=-l \\ l+m \text{ odd}}}^l \lambda_l
^m \left(y_l^m(\Vx) \overline{y_l^m}(\Vy) + \overline{y_l^m}(\Vx)y_l^m(\Vy) 
\right).\\
\label{eq:KW}
K_{\OW}(\Vx,\Vy) &:=& \frac{1}{4\pi} \frac{1}{\norm{\Vx-\Vy}{}^3} =
\sum_{l=0}^{\infty}\sum_{\substack{m=-l \\ l+m \text{ odd}
}}^l  \frac{1}{\lambda_l^m} \left( \dfrac{y_l^m(\Vx)}{\omega(\Vx)} \dfrac{
\overline{y_l^m}(\Vy)}{\omega(\Vy)} + \dfrac{\overline{y_l^m}(\Vx)}{\omega(\Vx)} 
\dfrac{y_l^m(\Vy)}{\omega(\Vy)} \right).\\
\label{eq:KmW}
K_{\mOW}(\Vx,\Vy) &:=& \sum_{l=0}^{\infty}\sum_{\substack{m=-l \\ l+m \text{ 
even}}}^l  \frac{4}{\lambda_l^m} \left( \dfrac{y_l^m(\Vx)}{\omega(\Vx)}\dfrac{
\overline{y_l^m}(\Vy)}{\omega(\Vy)}+\dfrac{\overline{y_l^m}(\Vx)}{\omega(\Vx)} 
\dfrac{y_l^m(\Vy)}{\omega(\Vy)} \right).
\end{eqnarray}
\normalsize
\end{prop}
\begin{proof}
The proof for the standard BIOs kernels follows from Propositions~\ref{prop:Wolfe} 
and \ref{prop:Krenk} combined with the orthogonality of the PSHs \eqref{eq:orthylm}. From \eqref{eq:WolfeV}, we can additionally find a key relationship fulfilled by 
its unique inverse operator $\mOW = \OV^{-1}$ over 
$\D_1$
\begin{equation}
\label{eq:mOWevp}
(\mOW y_l^m)(\Vx) = \frac{4}{\lambda_l^m} \dfrac{y_l^m(\Vx)}{\omega(\Vx)},
\end{equation}
and use again the orthogonality property to get \eqref{eq:KmW}. The result for 
$K_{\mOV}$ can be obtained analogously from \eqref{eq:KrenkW}. 
\end{proof}

\begin{definition}[{{\bf Weighted $L^2$-spaces} \cite[sl.~36--37]{NED13}
		\cite[Def.2.7.8--2.7.9]{RAM16}}]
	Let us define the spaces $L^2_{1/\omega}(\D_1)$ and $L^2_{\omega}(\D_1)$ as 
	the $L^2$-spaces induced by the weighted inner products
	\begin{align}
	(u,v)_{1/\omega} = \int_{\D_1} u(\Vx) \overline{v(\Vx)} \omega(\Vx)^{-1} d
	\D_1(\Vx),
	\end{align}
	and 
	\begin{align}
	(u,v)_{\omega} = \int_{\D_1} u(\Vx) \overline{v(\Vx)} \omega(\Vx) d\D_1(\Vx),
	\end{align}
	respectively.
\end{definition}
\begin{prop}[{\cite[sl.~36--37]{NED13}\cite[Prop~2.7.8--2.7.9]{RAM16}}]
	\label{prop:L2weightedbasis}
	We have that 
	\begin{itemize}
		\item $\lbrace y_l^m \: : \: l+m \text{ odd} \rbrace$ and 
		$\lbrace y_l^m \: : \: l+m \text{ even} \rbrace$ are both orthogonal bases for 
		$L^2_{1/\omega}$.\\
		
		\item $\lbrace {\omega}^{-1}{y_l^m} \: : \: l+m \text{ odd} 
		\rbrace$ and $\lbrace {\omega}^{-1}{y_l^m} \: : \: l+m \text{ even} \rbrace$ 
		are both orthogonal bases for $L^2_{\omega}$.
	\end{itemize}
\end{prop}

\begin{prop}[{\cite[Sect.~2.7.4]{RAM16}}]
	\label{prop:bases}
	\begin{enumerate}[(i)]
		\item \label{whp} $u\in\wH^{1/2}(\D_1)$ can be expanded in the basis 
		$\lbrace y_l^m \: : \: l+m \text{ \bf odd} \rbrace$ of $L^2_{1/\omega}$:
		\begin{equation}
		u(\Vx) = \sum_{l=0}^{\infty} \sum_{m=-l}^{l} u_l^m y_l^m(\Vx), \quad 
		u_l^m = (u,y_l^m)_{1/\omega},\quad l+m \text{ odd}.
		\end{equation}
		\item $g\in H^{1/2}(\D_1)$ can be expanded in the basis $\lbrace y_l^m \: : 
		\: l+m \text{ \bf even} \rbrace$ of $L^2_{1/\omega}$ :
		\begin{equation}
		g(\Vx) = \sum_{l=0}^{\infty} \sum_{m=-l}^{l} g_l^m y_l^m(\Vx), \quad g_l^m 
		= (g,y_l^m)_{1/\omega}, \quad l+m \text{ even}.
		\end{equation}
		\item $\upsilon\in H^{-1/2}(\D_1)$ can be expanded in the basis 
		$\lbrace y_l^m\omega^{-1} \: : \: l+m \text{ \bf odd} \rbrace$ of $L^2_{\omega
		}$:
		\begin{equation}
		\upsilon(\Vx) = \sum_{l=0}^{\infty} \sum_{m=-l}^{l}\upsilon_l^m \frac{y_l^m(
			\Vx)}{\omega(\Vx)}, \quad \upsilon_l^m = (\upsilon,{
			\omega}^{-1}{y_l^m})_{\omega}, \quad l+m \text{ odd}.
		\end{equation}
		\item $\sigma\in\wH^{-1/2}(\D_1)$ can be expanded in the basis 
		$\lbrace y_l^m\omega^{-1} \: : \: l+m \text{ \bf even} \rbrace$ of $L^2_{\omega
		}$:
		\begin{equation}
		\sigma(\Vx) = \sum_{l=0}^{\infty} \sum_{m=-l}^{l}\sigma_l^m \frac{y_l^m(\Vx)}{
			\omega(\Vx)}, \quad \sigma_l^m = (\sigma,{\omega}^{-1}y_l^m)_{\omega}, 
		\quad l+m \text{ even}.
		\end{equation}
	\end{enumerate}
\end{prop}
\begin{proof} 
	Since the bilinear form associated to $\OW$ is symmetric, elliptic and 
	continuous, it induces an energy inner product on $\wH^{1/2}(\D_1)$. Then, 
	the proof of (\ref{whp}) boils down to showing that for $u \in \wH^{1/2}(\D_1)$:
	\begin{equation*}
	\dual{W u}{ y_l^m}_{\D_1}=0, \: \forall l+m \text{ odd } \quad \Leftrightarrow 
	\quad u \equiv 0. 
	\end{equation*}
	Using the symmetry of the bilinear form and \eqref{eq:KrenkW}, we derive
	\begin{align*}
	\dual{W u}{y_l^m}_{\D_1} = \frac{1}{\lambda_l^m} \dual{u}{\frac{y_l^m}{
			\omega}}_{\D_1} = \frac{1}{\lambda_l^m} (u,y_l^m)_{1/\omega},
	\end{align*}
	which is zero if and only if  $u \equiv 0 $ because $\lbrace y_l^m \: : \: 
	l+m \text{ \bf odd} \rbrace$ is an orthogonal basis for $L^2_{1/\omega}$. The remaining three cases follow by analogy.
\end{proof}
~\\

	One can prove that {\bf odd} PSHs can be factored as
	\begin{equation*}
	y_l^m(\Vx) = e^{im\theta_x}\omega(\Vx) \Psi(\Vx), \quad l+m \text{ odd},
	\end{equation*}
	where $\Psi$ is a polynomial function (Combine \cite[Eq.~14.3.21]{NIST:DLMF}, 
	\cite[Eq.~3.2(7)]{BAE53}, and \cite[Eq.~10.9(22)]{BAE53}). However, this is not true when $l+m$ is {\bf even}. In that case, the radial 
	part of $y_l^m$ is already a polynomial (since \cite[10.9(21)]{BAE53} 
	holds instead of \cite[10.9(22)]{BAE53}). This property confirms that the basis functions of our four fractional Sobolev 
	spaces have the correct behaviour. Namely, when $l+m$ is odd, $y_l^m\sim\omega$
	near the boundary and belongs to $\wH^{1/2}(\D_1)$; when $l+m$ is even, $y_l^m\omega^{-1}\sim
	\omega^{-1}$ near $\partial\D_1$ and lies in $\wH^{-1/2}(\D_1)$; while the basis functions 
	of $H^{1/2}(\D_1)$ and $H^{-1/2}(\D_1)$ have no singular behaviour.

\begin{definition}[{{\bf Kinetic moments on $\D_1$}\cite[Eq.~(51)]{NED13}}]
	Define the operators $\oL_+$ and $\oL_-$ of derivation over $\D_1$ as
	\begin{align}
	\oL_{\pm} u \ :=\ e^{\pm i\theta} \left(\pm\frac{\partial u}{\partial r} + i 
	\frac{1}{r} \frac{\partial u}{\partial\theta} \right).
	\end{align}
\end{definition}

\begin{prop}[{{\bf Properties of the kinetic moments over $\D_1$}\cite[sl.~21,
		Eq.~(80)]{NED13}, \cite[Prop.~2.7.7, Cor.~2.7.2, Lemma~2.7.3]{RAM16}}]
	Let $u,v\in C^{\infty}(\D_1)$, and $\Vx,\Vy\in\D_1$. The kinetic moments 
	satisfy over $\D_1$
	\begin{equation}
	\label{eq:curldotcurl}
	\curl_{\D_1,\Vx}u(\Vx)\cdot\curl_{\D_1,\Vy}v(\Vy) = -\frac{1}{2}\left(\oL_{+,
		\Vx} u(\Vx) \oL_{-,\Vy} \overline{v(\Vy)} + \oL_{-,\Vx} u(\Vx) \oL_{+, \Vy} 
	\overline{v(\Vy)} \right), 
	\end{equation}
	together with $\oL_+^* = \oL_-$, and $\oL_-^* = \oL_+$.
	Moreover, when applied to PSHs, we get
	\begin{align}
	\label{eq:Lpy}
	\oL_+ y_l^m(\Vx) = \sqrt{(l-m)(l+m+1)} \frac{y_l^{m+1}(\Vx)}{\omega(\Vx)},\\
	\label{eq:Lmy}
	\oL_- y_l^m(\Vx) = \sqrt{(l+m)(l-m+1)} \frac{y_l^{m-1}(\Vx)}{\omega(\Vx)}.
	\end{align}
\end{prop}

\begin{rmk}
	Due to Proposition~\ref{prop:bases}, \eqref{eq:Lpy} 
	and \eqref{eq:Lmy} can also be interpreted as: $\oL_{\pm}$ maps $H^{1/2}(\D_1)
	$ to $H^{-1/2}(\D_1)$, and $\wH^{1/2}(\D_1)$ to $\wH^{-1/2}(\D_1)$.
\end{rmk}

\emph{Proof of \autoref{prop:Vinvvar}}
We begin our proof by introducing the following recursion formula
\begin{equation}
\label{eq:recgamma}
\frac{4}{\lambda_l^m} = \frac{1}{2}\left[ (l+m)(l-m+1)\lambda_l^{m-1} 
+ (l-m)(l+m+1)\lambda_l^{m+1} \right],
\end{equation}
which can be verified by direct computations using the multiplicative property 
of the Gamma function,i.e., $\mathsf{\Gamma}(z+1)=z\mathsf{\Gamma}(z)$. Plugging this recursion formula into \eqref{eq:mOWevp} gives
\begin{align}
\label{eq:rec2}
(\mOW y_l^m)(\Vx) &= \frac{4}{\lambda_l^m} \frac{y_l^m(\Vx)}{\omega(\Vx)} 
\nonumber \\ &= \frac{1}{2}\left[ (l+m)(l-m+1)\lambda_l^{m-1} 
+ (l-m)(l+m+1)\lambda_l^{m+1} \right] \frac{y_l^m(\Vx)}{\omega(\Vx)}.
\end{align}
From \eqref{eq:Lpy} and \eqref{eq:Lmy}, it is clear that
\begin{align}
\oL_+ y_l^{m-1}(\Vx) = \sqrt{(l-m+1)(l+m)} \frac{y_l^{m}(\Vx)}{\omega(\Vx)},\\
\oL_- y_l^{m+1}(\Vx) = \sqrt{(l+m+1)(l-m)} \frac{y_l^{m}(\Vx)}{\omega(\Vx)}.
\end{align}
Moreover, by unicity of $\OW^{-1}=\mOV$ (\emph{cf.~}\cite[Prop.~2.2]{HJU16}), it must hold that 
\begin{align}
\left(\mOV \frac{y_l^{m\pm1}}{\omega} \right)(\Vx) = \lambda_l^{m\pm1} y_l^{m\pm1}(\Vx), 
\quad l+m\pm1 \text{ odd}.
\end{align}
Then, combining all these ingredients, it is clear that for $(l,m)\neq(0,0)$, 
$l+m$ even\footnote{It holds $\dfrac{4}{\lambda_0^0} = \dfrac{1}{2}\left( (0)(1)\lambda_0^{-1} 
	+ (0)(1)\lambda_0^1 \right) = \dfrac{4}{\pi} \Gamma(0)0$
	where it is crucial that $\lim_{s\rightarrow0} \Gamma(s)s=1$. In order to go 
	from \eqref{eq:recgamma} to \eqref{eq:rec2} one actually ``splits this 
	limit'' and breaks the identity.}, our expression is equivalent to
\begin{equation*}
\label{eq:ibpspec}
(\mOW y_l^m)(\Vx) = \frac{1}{2}\left( \oL_+ \mOV \oL_- y_l^m(\Vx) 
+ \oL_- \mOV \oL_+ y_l^m(\Vx) \right).
\end{equation*}
It follows that the associated bilinear form is
\begin{align}
\label{eq:ibpspec2}
\dual{\mOW y_{l_1}^{m_1}}{y_{l_2}^{m_2}}_{\D_1} &= \frac{1}{2}\left(\dual{
	\oL_+ \mOV\oL_- y_{l_1}^{m_1}}{y_{l_2}^{m_2}}_{\D_1}+\dual{\oL_- \mOV\oL_+ 
	y_{l_1}^{m_1} }{y_{l_2}^{m_2}}_{\D_1}\right) \nonumber \\
&= \frac{1}{2}\left(\dual{\mOV\oL_- y_{l_1}^{m_1}}{\oL_- y_{l_2}^{m_2}}_{\D_1}
+\dual{\mOV\oL_+ y_{l_1}^{m_1}}{\oL_+y_{l_2}^{m_2}}_{\D_1}\right),
\end{align}
for $(l_1,m_1)\neq(0,0)$ and $(l_2,m_2)\neq(0,0)$.

Finally, \eqref{eq:curldotcurl} and 
\begin{align}
\overline{\oL}_\pm = -\oL_{\mp},
\end{align}
imply that \eqref{eq:ibpspec2} can be rewritten as the desired formula.
\vbox{\hrule height0.6pt\hbox{%
		\vrule height1.3ex width0.6pt\hskip0.8ex
		\vrule width0.6pt}\hrule height0.6pt
}\\

It is worth noticing that the condition $(l,m)\neq(0,0)$ only 
excludes the constants, characterized by $y_0^0$. Due to the orthogonality 
\eqref{eq:orthylm}, this space is defined by
\begin{equation*}
H^{1/2}_*(\D_1) = \lbrace v \in H^{1/2}(\D_1) \,:\, \dual{
	v}{\omega^{-1}}_{\D_1}=0 \rbrace,
\end{equation*}
as introduced in Proposition \ref{prop:Vinvvar}.

\section{Conclusion}
\label{sec:conclusion}

We have introduced new modified hypersingular and weakly singular integral 
operators that supply the exact inverses for the standard ones on disks. 
Moreover, we provide variational forms for these modified BIOs that are 
amenable to standard Galerkin boundary element discretization.

We consider the disk $\D_a$ to be the canonical shape for more general screens 
$\Gamma$ and therefore an important starting point of our analysis. From this 
perspective, the new Calder\'on-type identities that we have shown over $\D_a$ 
are suitable for Calder\'on preconditioning on three-dimensional parametrized 
screens, in the same fashion as in 2D \cite{HJU14}. The applicability of these Calder\'on-type identities to construct $h$-independent 
preconditioners will be discussed in a upcoming article focusing on the numerical 
aspect of these new modified BIOs \cite{HJU17c, HJU16, HJU17}. Though we developed our analysis only for the Laplace equation over disks 
with Dirichlet and Neumann boundary conditions, extensions to 
other elliptic equations, like the scalar Helmholtz equation, can be immediately  achieved via perturbation arguments.


\end{document}